\documentclass[english,oneside,a4paper,11pt]{article}
\usepackage{amssymb}
\usepackage{latexsym}
\usepackage[english]{babel}
\usepackage{amsmath}
\usepackage{amsfonts}
\usepackage{amsbsy}
\usepackage[mathscr]{eucal}
\usepackage[latin1]{inputenc}
\usepackage{verbatim}
\usepackage{mathrsfs}
\usepackage{graphicx}
\usepackage{float}
\usepackage{bbm}
\usepackage{lscape}
\usepackage{mathdots}
\usepackage{empheq}
\usepackage{color}
\usepackage{array}

\providecommand{\abs}[1]{\left\lvert#1 \right\rvert}

\newtheorem{thr}{Theorem}
\newtheorem{lem}{Lemma}

\newenvironment{rem}
{\begin{trivlist}\item[\hskip%
\labelsep{{\it \noindent Remark}}]}{\hfill
\end{trivlist}}

\newenvironment{proof}
{\begin{trivlist}\item[\hskip%
\labelsep{\it \noindent Proof.}]}{\hfill $\square$ \rm
\end{trivlist}}

\newenvironment{proofTheorem1}
{\begin{trivlist}\item[\hskip%
\labelsep{{\it \noindent Proof of Theorem 1.}}]}{\hfill $\square$
\end{trivlist}}

\newenvironment{proofTheorem2}
{\begin{trivlist}\item[\hskip%
\labelsep{{\it \noindent Proof of Theorem 2.}}]}{\hfill $\square$
\end{trivlist}}

\numberwithin{equation}{section}

\allowdisplaybreaks

\begin{document}
\begin{center}
{\huge {\bf On the spectral properties of real} \\ {\bf anti-tridiagonal Hankel matrices}} \\
\vspace{1.0cm}
{\Large João Lita da Silva\footnote{\textit{E-mail address:} \texttt{jfls@fct.unl.pt}; \texttt{joao.lita@gmail.com}}} \\
\vspace{0.1cm}
\textit{Department of Mathematics and GeoBioTec \\ Faculty of Sciences and Technology \\
NOVA University of Lisbon \\ Quinta da Torre, 2829-516 Caparica,
Portugal}
\end{center}

\bigskip

\bigskip

\begin{abstract}
In this paper we express the eigenvalues of real anti-tridiagonal Hankel matrices as the zeros of given rational functions. We still derive eigenvectors for these structured matrices at the expense of prescribed eigenvalues.
\end{abstract}

\bigskip


\bigskip


\bigskip

\section{Introduction}\label{sec:1}

\indent

Recently, some authors have computed the eigenvalues and eigenvectors for a sort of Hankel matrices thus obtaining its eigendecomposition (see \cite{Gutierrez14}, \cite{Gutierrez16}, \cite{Lita16}, \cite{Rimas13a}, \cite{Rimas13b}, \cite{Wu10}, \cite{Yin08} among others). In contrast to the tridiagonal Toeplitz matrices case where the eigenvalues and eigenvectors are well-known (see, for instance, \cite{Meyer00}), a possible closed-form expression for the eigenvalues and eigenvectors of general anti-tridiagonal Hankel matrices is still to be found.

The aim of this short note is to give a contribution in that search. Specifically, we shall present an eigenvalue localization tool for real anti-tridiagonal Hankel matrices, providing also associated eigenvectors. To achieve our purpose, we shall use eigendecompositions of anti-circulant matrices available in \cite{Karner03} to ensure a decomposition for the matrices under study in a first step, and results concerning to sum of (rank one) matrices in a final stage to obtain all formulae.

The rational functions exhibited in this paper to locate eigenvalues of real anti-tridiagonal Hankel matrices, as well as the expressions for its eigenvectors, are given in explicit form which, on one hand, can be easily evaluated in computer programs, and, on the other, are useful for further theoretical investigations in this subject.

\section{Main results}\label{sec:2}

Let $n$ be a positive integer and consider the following $(n + 2) \times (n + 2)$ anti-tridiagonal Hankel matrix

\begin{equation}\label{eq:2.1}
\mathbf{H}_{n+2} = \left[
\begin{array}{ccccccc}
  0       & \ldots  & \ldots  & \ldots  & 0       & a       & c \\
  \vdots  &         &         & \iddots & a       & c       & b \\
  \vdots  &         & \iddots & \iddots & c       & b       & 0 \\
  \vdots  & \iddots & \iddots & \iddots & \iddots & \iddots & \vdots \\
  0       & a       & c       & \iddots & \iddots &         & \vdots \\
  a       & c       & b       & \iddots &         &         & \vdots \\
  c       & b       & 0       & \ldots  & \ldots  & \ldots  & 0
\end{array}
\right],
\end{equation}
where $a,b,c$ are real numbers. Throughout, we shall set
\begin{equation}\label{eq:2.2}
\omega := \mathrm{e}^{\frac{2\pi \mathrm{i}}{n + 2}},
\end{equation}
where $\mathrm{i}$ denotes the imaginary unit.

\subsection{Eigenvalue localization for $\mathbf{H}_{n+2}$}

\indent

Our first statement is an eigenvalue localization theorem for matrices of the form \eqref{eq:2.1}.

% THEOREM 1

\begin{thr}\label{thr1}
Let $n$ be a positive integer, $a,b,c$ real numbers, $\omega$ given by \eqref{eq:2.2},
\begin{subequations}
\begin{gather}
\lambda_{k} := b  + a \omega^{-nk} + c \omega^{-(n + 1)k}, \quad k = 0,1,\ldots,n+1  \label{eq:2.3a} \\
\theta_{k} := \mathrm{arg}(\lambda_{k}), \quad k = 0,1,\ldots,n+1 \label{eq:2.3b} \\
F_{n}(t;\alpha,\beta) := \tfrac{1}{t - \lambda_{0}} + 2 \sum_{k=1}^{n} \left[\tfrac{\cos\left(\frac{\theta_{k}}{2} - \alpha k \right) \cos\left(\frac{\theta_{k}}{2} - \beta k \right)}{t - \abs{\lambda_{k}}} + \tfrac{\sin\left(\frac{\theta_{k}}{2} - \alpha k \right) \sin\left(\frac{\theta_{k}}{2} - \beta k \right)}{t + \abs{\lambda_{k}}} \right] \label{eq:2.3c} \\
G_{n}(t;\alpha,\beta) := \tfrac{1}{t - \lambda_{0}} + \tfrac{1}{t - \lambda_{n + 1}} + 2 \sum_{k=1}^{n} \left[\tfrac{\cos\left(\frac{\theta_{k}}{2} - \alpha k \right) \cos\left(\frac{\theta_{k}}{2} - \beta k \right)}{t - \abs{\lambda_{k}}} + \tfrac{\sin\left(\frac{\theta_{k}}{2} - \alpha k \right) \sin\left(\frac{\theta_{k}}{2} - \beta k \right)}{t + \abs{\lambda_{k}}} \right] \label{eq:2.3d}
\end{gather}
\end{subequations}

\medskip

\begin{subequations}
\noindent \textnormal{(a)} If $n$ is odd, then the eigenvalues of $\mathbf{H}_{n+2}$ that are not of the form $\lambda_{0}$, $\lvert \lambda_{k} \rvert$, $-\lvert \lambda_{k} \rvert$ $k=1,\ldots,\frac{n+1}{2}$ are precisely the zeros of the function
\begin{equation}\label{eq:2.4a}
\begin{split}
& f(t) = 1 + \tfrac{a}{n + 2} F_{\frac{n + 1}{2}} \left(t;\tfrac{2\pi}{n + 2},\tfrac{2\pi}{n + 2} \right) + \tfrac{b}{n + 2} F_{\frac{n + 1}{2}}(t;0,0) + \\
& \qquad \tfrac{ab}{(n + 2)^{2}} F_{\frac{n + 1}{2}} \left(t;\tfrac{2\pi}{n + 2}, \tfrac{2\pi}{n + 2} \right) F_{\frac{n + 1}{2}}(t;0,0) - \tfrac{ab}{(n + 2)^{2}} F_{\frac{n + 1}{2}}^{2}\left(t;0,\tfrac{2\pi}{n + 2} \right).
\end{split}
\end{equation}
Moreover, if $\mu_{1} \leqslant \mu_{2} \leqslant \ldots \leqslant \mu_{n+2}$ are the eigenvalues of $\mathbf{H}_{n+2}$ and $\lambda_{0}$, $\lvert \lambda_{k} \rvert$, $-\lvert \lambda_{k} \rvert$ $k=1,\ldots,\frac{n+1}{2}$ are arranged in non-decreasing order as $d_{1} \leqslant d_{2} \leqslant \ldots \leqslant d_{n+2}$ then
\begin{equation}\label{eq:2.4b}
d_{k} + \min\left\{0, -a, -b \right\} \leqslant \mu_{k} \leqslant d_{k} + \max\left\{0, -a, -b \right\}, \quad k=1,\ldots,n+2.
\end{equation}
\end{subequations}

\medskip

\begin{subequations}
\noindent \textnormal{(b)} If $n$ is even, then the eigenvalues of $\mathbf{H}_{n+2}$ that are not of the form $\lambda_{0}$, $\lvert \lambda_{k} \rvert$, $\lambda_{\frac{n}{2} + 1}$, $-\lvert \lambda_{k} \rvert$ $k=1,\ldots,\frac{n}{2}$ are precisely the zeros of the function
\begin{equation}\label{eq:2.5a}
\begin{split}
& g(t) = 1 + \tfrac{a}{n + 2} G_{\frac{n}{2}} \left(t;\tfrac{2\pi}{n + 2},\tfrac{2\pi}{n + 2} \right) + \tfrac{b}{n + 2} G_{\frac{n}{2}}(t;0,0) + \\
& \qquad \tfrac{ab}{(n + 2)^{2}} G_{\frac{n}{2}} \left(t;\tfrac{2\pi}{n + 2}, \tfrac{2\pi}{n + 2} \right) G_{\frac{n}{2}}(t;0,0) - \tfrac{ab}{(n + 2)^{2}} \left[G_{\frac{n}{2}}\left(t;0,\tfrac{2\pi}{n + 2} \right) - \tfrac{2}{t - \lambda_{\frac{n}{2} + 1}} \right]^{2}.
\end{split}
\end{equation}
Moreover, if $\nu_{1} \leqslant \nu_{2} \leqslant \ldots \leqslant \nu_{n+2}$ are the eigenvalues of $\mathbf{H}_{n+2}$ and $\lambda_{0}$, $\lvert \lambda_{k} \rvert$, $\lambda_{\frac{n}{2} + 1}$, $-\lvert \lambda_{k} \rvert$ $k=1,\ldots,\frac{n}{2}$ are arranged in non-decreasing order as $d_{1} \leqslant d_{2} \leqslant \ldots \leqslant d_{n+2}$ then
\begin{equation}\label{eq:2.5b}
d_{k} + \min\left\{0, -a, -b \right\} \leqslant \nu_{k} \leqslant d_{k} + \max\left\{0, -a, -b \right\}, \quad k=1,\ldots,n+2.
\end{equation}
\end{subequations}
\end{thr}

%\begin{rem}
%Let us note that in \cite{Elouafi11} and \cite{Solary13}, rational functions were obtained in order to locate the eigenvalues of pentadiagonal and heptadiagonal symmetric Toeplitz matrices, respectively. Thus, our expressions \eqref{eq:2.4a} and \eqref{eq:2.5a} are in line with these results despite its complexity (inherent to matrices whose the entries are formed by anti-diagonals).
%Unlike banded Toeplitz matrices where several formulas and algorithms exist to compute the characteristic polynomial, for matrices having antidiagonals, in general, the determination of the eigenvalues of antidiagonals requires tedious calculation (\cite{Karner03}, page $303$).
%\end{rem}

\subsection{Eigenvectors of $\mathbf{H}_{n+2}$}

\indent

Owning the eigenvalues of $\mathbf{H}_{n+2}$ in \eqref{eq:2.1}, we are able to determine the corresponding eigenvectors.

%THEOREM 2

\begin{thr}\label{thr2}
Let $n$ be an integer, $a,b,c$ real numbers such that $a \neq 0$, and $\lambda_{k}$, $\theta_{k}$ $(k = 0,1,\ldots,n+1)$, $F_{n}(t;\alpha,\beta)$, $G_{n}(t;\alpha,\beta)$ be given by \eqref{eq:2.3a}, \eqref{eq:2.3b}, \eqref{eq:2.3c}, \eqref{eq:2.3d}, respectively.

\medskip

\noindent \textnormal{(a)} If $n$ is odd, the zeros $\mu_{1}, \mu_{2}, \ldots, \mu_{n+2}$ of \eqref{eq:2.4a} are not of the form $\lambda_{0}$, $\lvert \lambda_{k} \rvert$, $-\lvert \lambda_{k} \rvert$, $k=1,\ldots,\frac{n+1}{2}$ and $b \, F_{\frac{n + 1}{2}}(\mu_{m},0,0) + n + 2 \neq 0$, then
\begin{equation}\label{eq:2.6}
\mathbf{u}(\mu_{m}) = \left[
\begin{array}{c}
  F_{\frac{n+1}{2}}\left(\mu_{m};0,\frac{2\pi}{n + 2} \right) - \frac{b \, F_{\frac{n + 1}{2}} \left(\mu_{m},0,\frac{2\pi}{n + 2} \right) F_{\frac{n+1}{2}}\left(\mu_{m};0,0 \right)}{b \, F_{\frac{n + 1}{2}}(\mu_{m},0,0) + n + 2} \\[10pt]
  F_{\frac{n+1}{2}}\left(\mu_{m};-\frac{2\pi}{n + 2},\frac{2\pi}{n + 2} \right) - \frac{b \, F_{\frac{n + 1}{2}} \left(\mu_{m},0,\frac{2\pi}{n + 2} \right) F_{\frac{n+1}{2}}\left(\mu_{m};-\frac{2\pi}{n + 2},0 \right)}{b \, F_{\frac{n + 1}{2}}(\mu_{m},0,0) + n + 2} \\[5pt]
  \vdots \\[10pt]
  F_{\frac{n+1}{2}}\left(\mu_{m};\frac{2(1 - k)\pi}{n + 2},\frac{2\pi}{n + 2} \right) - \frac{b \, F_{\frac{n + 1}{2}} \left(\mu_{m},0,\frac{2\pi}{n + 2} \right) F_{\frac{n+1}{2}}\left(\mu_{m};\frac{2(1 - k)\pi}{n + 2},0 \right)}{b \, F_{\frac{n + 1}{2}}(\mu_{m},0,0) + n + 2} \\[10pt]
  \vdots \\[10pt]
  F_{\frac{n+1}{2}}\left(\mu_{m};-\frac{2(n + 1)\pi}{n + 2},\frac{2\pi}{n + 2} \right) - \frac{b \, F_{\frac{n + 1}{2}} \left(\mu_{m},0,\frac{2\pi}{n + 2} \right) F_{\frac{n+1}{2}}\left(\mu_{m};-\frac{2(n + 1)\pi}{n + 2},0 \right)}{b \, F_{\frac{n + 1}{2}}(\mu_{m},0,0) + n + 2}
\end{array}
\right]
\end{equation}
is an eigenvector of $\mathbf{H}_{n+2}$ associated to $\mu_{m}$, $m=1,2,\ldots, n+2$.

\medskip

\noindent \textnormal{(b)} If $n$ is even, the zeros $\nu_{1}, \nu_{2}, \ldots, \nu_{n+2}$ of \eqref{eq:2.5a} are not of the form $\lambda_{0}$, $\lvert \lambda_{k} \rvert$, $\lambda_{\frac{n}{2} + 1}$, $-\lvert \lambda_{k} \rvert$, $k=1,\ldots,\frac{n}{2}$ and $b \, G_{\frac{n}{2}}(\nu_{m},0,0) + n + 2 \neq 0$, then
\begin{equation}\label{eq:2.7}
\begin{split}
&\mathbf{v}(\nu_{m}) = \left[
\begin{array}{@{\hspace{1pt}}c@{\hspace{1pt}}}
  G_{\frac{n}{2}}\left(\nu_{m};0,\frac{2\pi}{n + 2} \right) - \frac{2}{\nu_{m} - \lambda_{\frac{n}{2} + 1}} - \frac{b \, G_{\frac{n}{2}}\left(\nu_{m};0,0 \right) \left[G_{\frac{n}{2}}\left(\nu_{m};0,\frac{2\pi}{n + 2} \right) - \frac{2}{\nu_{m} - \lambda_{\frac{n}{2} + 1}} \right]}{b \, G_{\frac{n}{2}}(\nu_{m};0,0) + n + 2} \\[10pt]
  G_{\frac{n}{2}}\left(\nu_{m};-\frac{2\pi}{n + 2},\frac{2\pi}{n + 2} \right) - \frac{b \left[G_{\frac{n}{2}}\left(\nu_{m};-\frac{2\pi}{n + 2},0 \right) - \frac{2}{\nu_{m} - \lambda_{\frac{n}{2} + 1}} \right] \left[G_{\frac{n}{2}}\left(\nu_{m};0,\frac{2\pi}{n + 2} \right) - \frac{2}{\nu_{m} - \lambda_{\frac{n}{2} + 1}} \right]}{b \, G_{\frac{n}{2}}(\nu_{m};0,0) + n + 2} \\[5pt]
  \vdots \\[10pt]
  G_{\frac{n}{2}}\left(\nu_{m};\frac{2(1 - k)\pi}{n + 2},\frac{2\pi}{n + 2} \right) - \frac{1 - (-1)^{k}}{\nu_{m} - \lambda_{\frac{n}{2} + 1}} - \frac{b \left[G_{\frac{n}{2}}\left(\nu_{m};\frac{2(1 - k)\pi}{n + 2},0 \right) - \frac{1 + (-1)^{k}}{\nu_{m} - \lambda_{\frac{n}{2} + 1}} \right] \left[G_{\frac{n}{2}}\left(\nu_{m};0,\frac{2\pi}{n + 2} \right) - \frac{2}{\nu_{m} - \lambda_{\frac{n}{2} + 1}} \right]}{b \, G_{\frac{n}{2}}(\nu_{m};0,0) + n + 2} \\[10pt]
  \vdots \\[10pt]
  G_{\frac{n}{2}}\left(\nu_{m};-\frac{2(n + 1)\pi}{n + 2},\frac{2\pi}{n + 2} \right) - \frac{b \left[G_{\frac{n}{2}}\left(\nu_{m};-\frac{2(n + 1)\pi}{n + 2},0 \right) - \frac{2}{\nu_{m} - \lambda_{\frac{n}{2} + 1}} \right] \left[G_{\frac{n}{2}} \left(\nu_{m};0,\frac{2\pi}{n + 2} \right) - \frac{2}{\nu_{m} - \lambda_{\frac{n}{2} + 1}} \right]}{b \, G_{\frac{n}{2}}(\nu_{m};0,0) + n + 2}
\end{array}
\right]
\end{split}
\end{equation}
is an eigenvector of $\mathbf{H}_{n+2}$ associated to $\nu_{m}$, $m=1,2,\ldots, n+2$.
\end{thr}

\begin{rem}
We point out that if $a = 0$ then expressions \eqref{eq:2.6} and \eqref{eq:2.7} should be replaced by
\begin{equation*}
\mathbf{u}(\mu_{m}) = \left[
\begin{array}{c}
  F_{\frac{n+1}{2}}\left(\mu_{m};0,0 \right) \\
  F_{\frac{n+1}{2}}\left(\mu_{m};-\frac{2\pi}{n + 2},0 \right) \\
  \vdots \\
  F_{\frac{n+1}{2}}\left(\mu_{m};-\frac{2(n + 1)\pi}{n + 2},0 \right)
\end{array}
\right]
\end{equation*}
and
\begin{equation*}
\mathbf{v}(\nu_{m}) = \left[
\begin{array}{c}
  G_{\frac{n}{2}}\left(\nu_{m};0,0 \right) \\
  G_{\frac{n}{2}}\left(\nu_{m};-\frac{2\pi}{n + 2},0 \right) - \frac{2}{\nu_{m} - \lambda_{\frac{n}{2} + 1}} \\
  \vdots \\
  G_{\frac{n}{2}} \left(\nu_{m};\frac{2(1 - k)\pi}{n + 2},0 \right) - \frac{1 + (-1)^{k}}{\nu_{m} - \lambda_{\frac{n}{2} + 1}} \\
  \vdots \\
  G_{\frac{n}{2}}\left(\nu_{m};-\frac{2(n + 1)\pi}{n + 2},0 \right) - \frac{2}{\nu_{m} - \lambda_{\frac{n}{2} + 1}}
\end{array}
\right]
\end{equation*}
respectively, provided that $b \neq 0$. Of course, if $a = b = 0$ then an eigenvalue decomposition of the exchange matrix is already known (see \cite{Gutierrez14}).
\end{rem}

\section{Lemmas and proofs}

Let $n$ be a positive integer. Consider the following real anti-circulant $(n+2) \times (n+2)$ matrix

\begin{equation}\label{eq:3.1}
\mathbf{A}_{n+2} := \left[
\begin{array}{ccccccc}
  b       & 0       & \ldots  & \ldots  & 0       & a       & c \\
  0       &         &         & \iddots & a       & c       & b \\
  \vdots  &         & \iddots & \iddots & c       & b       & 0 \\
  \vdots  & \iddots & \iddots & \iddots & \iddots & \iddots & \vdots \\
  0       & a       & c       & \iddots & \iddots &         & \vdots \\
  a       & c       & b       & \iddots &         &         & 0 \\
  c       & b       & 0       & \ldots  &  \ldots & 0       & a
\end{array}
\right]
\end{equation}
and the $(n+2) \times (n+2)$ unitary discrete Fourier transform matrix $\boldsymbol{\Omega}_{n+2}$, that is, the matrix defined by
\begin{equation}\label{eq:3.2}
\left[\boldsymbol{\Omega}_{n+2} \right]_{k,\ell} := \frac{\omega^{(k - 1) (\ell - 1)}}{\sqrt{n + 2}}
\end{equation}
where $\omega$ is given by \eqref{eq:2.2}. Our first auxiliary result is an orthogonal decomposition for $\eqref{eq:3.1}$. We shall denote by $\ast$ the conjugate transpose of any complex matrix.

%LEMMA 1

\begin{lem}\label{lem1}
Let $n$ be a positive integer, $a,b,c$ real numbers, and $\omega$, $\lambda_{k}$, $\theta_{k}$, $k = 0,1,\ldots,n+1$ given by \eqref{eq:2.2}, \eqref{eq:2.3a}, \eqref{eq:2.3b}, respectively.

\medskip

\begin{subequations}
\noindent \textnormal{(a)} If $n$ is odd, then
\begin{equation}\label{eq:3.3a}
\mathbf{A}_{n+2} = \mathbf{P}_{n+2} \, \mathrm{diag} \left(\lambda_{0},\lvert\lambda_{1} \rvert,\ldots,\lvert \lambda_{\frac{n+1}{2}} \rvert,-\lvert\lambda_{\frac{n+1}{2}} \rvert, \ldots,-\lvert \lambda_{1} \rvert \right) \mathbf{P}_{n+2}^{\top}
\end{equation}
where $\mathbf{P}_{n+2}$ is the $(n+2) \times (n+2)$ orthogonal matrix defined by
\begin{equation}\label{eq:3.3b}
\left[\mathbf{P}_{n+2} \right]_{k,\ell} = \left\{
\begin{array}{l}
  \frac{1}{\sqrt{n + 2}}, \ \ell = 1 \\[8pt]
  \sqrt{\frac{2}{n + 2}} \cos \left[\frac{\theta_{\ell -1}}{2} + \frac{2(k - 1)(\ell - 1)\pi}{n + 2} \right], \ 1 < \ell \leqslant \frac{n + 3}{2} \\[8pt]
  \sqrt{\frac{2}{n + 2}} \sin \left[\frac{\theta_{n + 3 - \ell}}{2} + \frac{2(k - 1)(n + 3 - \ell)\pi}{n + 2} \right], \ \ell > \frac{n + 3}{2}
\end{array}
\right..
\end{equation}
\end{subequations}

\medskip

\begin{subequations}
\noindent \textnormal{(b)} If $n$ is even, then
\begin{equation}\label{eq:3.4a}
\mathbf{A}_{n+2} = \mathbf{Q}_{n+2} \mathrm{diag} \left(\lambda_{0},\lvert\lambda_{1} \rvert,\ldots,\lvert \lambda_{\frac{n}{2}} \rvert,\lambda_{\frac{n}{2} + 1},-\lvert\lambda_{\frac{n}{2}} \rvert, \ldots,-\lvert \lambda_{1} \rvert \right) \mathbf{Q}_{n+2}^{\top}
\end{equation}
where $\mathbf{Q}_{n+2}$ is the $(n+2) \times (n+2)$ orthogonal matrix whose entries are given by \begin{equation}\label{eq:3.4b}
\left[\mathbf{Q}_{n+2} \right]_{k,\ell} = \left\{
\begin{array}{l}
  \frac{1}{\sqrt{n + 2}}, \ \ell = 1 \\[8pt]
  \sqrt{\frac{2}{n + 2}} \cos \left[\frac{\theta_{\ell -1}}{2} + \frac{2(k - 1)(\ell - 1)\pi}{n + 2} \right], \ 1 < \ell \leqslant \frac{n}{2} + 1 \\[8pt]
  \frac{(-1)^{k-1}}{\sqrt{n + 2}}, \ \ell = \frac{n}{2} + 2 \\[8pt]
  \sqrt{\frac{2}{n + 2}} \sin \left[\frac{\theta_{n + 3 - \ell}}{2} + \frac{2(k - 1)(n + 3 - \ell)\pi}{n + 2} \right], \ \ell > \frac{n}{2} + 2
\end{array}
\right..
\end{equation}
\end{subequations}
\end{lem}

\begin{proof}
Let $n$ be a positive odd integer. According to Theorem 3.6 of \cite{Karner03} we have
\begin{equation*}
\mathbf{A}_{n+2} = \mathbf{P}_{n+2} \, \mathrm{diag} \left(\lambda_{0},\lvert\lambda_{1} \rvert,\ldots,\lvert \lambda_{\frac{n+1}{2}} \rvert,-\lvert\lambda_{\frac{n+1}{2}} \rvert, \ldots,-\lvert \lambda_{1} \rvert \right) \mathbf{P}_{n+2}^{\ast}
\end{equation*}
where
\begin{equation*}
\mathbf{P}_{n+2} = \boldsymbol{\Omega}_{n+2}^{\ast} \left[
\begin{array}{cc}
  1 & \mathbf{0} \\
  \mathbf{0} & \mathbf{R}_{n+1}
\end{array}
\right]
\end{equation*}
with $\boldsymbol{\Omega}_{n+2}$ given by \eqref{eq:3.2} and $\mathbf{R}_{n+1}$ the following $(n+1) \times (n+1)$ matrix

\begin{equation*}
\mathbf{R}_{n+1} = \frac{\sqrt{2}}{2} \left[
\begin{array}{cccccccc}
  \mathrm{e}^{-\frac{\mathrm{i} \theta_{1}}{2}} & 0 & \ldots & 0 & 0 & \ldots & 0 & \mathrm{i} \mathrm{e}^{-\frac{\mathrm{i} \theta_{1}}{2}} \\[5pt]
  0 & \mathrm{e}^{-\frac{\mathrm{i} \theta_{2}}{2}} & \ddots & \vdots & \vdots & \iddots & \mathrm{i} \mathrm{e}^{-\frac{\mathrm{i} \theta_{2}}{2}} & 0 \\[5pt]
  \vdots & \ddots & \ddots & 0 & 0 & \iddots & \iddots & \vdots \\
  0 & \ldots & 0 & \mathrm{e}^{-\frac{\mathrm{i} \theta_{(n+1)/2}}{2}} & \mathrm{i} \mathrm{e}^{-\frac{\mathrm{i} \theta_{(n+1)/2}}{2}} & 0 & \ldots & 0 \\[5pt]
  0 & \ldots & 0 & \mathrm{e}^{\frac{\mathrm{i} \theta_{(n+1)/2}}{2}} & -\mathrm{i} \mathrm{e}^{\frac{\mathrm{i} \theta_{(n+1)/2}}{2}} & 0 & \ldots & 0 \\[5pt]
  \vdots & \iddots & \iddots & 0 & 0 & \ddots & \ddots & \vdots \\[5pt]
  0 & \mathrm{e}^{\frac{\mathrm{i} \theta_{2}}{2}} & \iddots & \vdots & \vdots & \ddots & -\mathrm{i} \mathrm{e}^{\frac{\mathrm{i} \theta_{2}}{2}} & 0 \\[5pt]
  \mathrm{e}^{\frac{\mathrm{i} \theta_{1}}{2}} & 0 & \ldots & 0 & 0 & \ldots & 0 & -\mathrm{i} \mathrm{e}^{\frac{\mathrm{i} \theta_{1}}{2}}
\end{array}
\right].
\end{equation*}
Note that the first column of $\mathbf{P}_{n+2}$ has all components equal to $1/\sqrt{n + 2}$; their next $(n + 1)/2$ columns are
\begin{equation}\label{eq:3.5}
\sqrt{\frac{2}{n + 2}} \left[
\begin{array}{c}
  \mathrm{e}^{-\frac{\mathrm{i} \theta_{\ell}}{2}} + \mathrm{e}^{\frac{\mathrm{i} \theta_{\ell}}{2}} \\
  \overline{\omega} \, \mathrm{e}^{-\frac{\mathrm{i} \theta_{\ell}}{2}} + \overline{\omega}^{(n+1)} \, \mathrm{e}^{\frac{\mathrm{i} \theta_{\ell}}{2}} \\
  \overline{\omega}^{2} \, \mathrm{e}^{-\frac{\mathrm{i} \theta_{\ell}}{2}} + \overline{\omega}^{2(n+1)} \,  \mathrm{e}^{\frac{\mathrm{i} \theta_{\ell}}{2}} \\
  \vdots \\
  \overline{\omega}^{n+1} \, \mathrm{e}^{-\frac{\mathrm{i} \theta_{\ell}}{2}} + \overline{\omega}^{(n+1)^{2}} \, \mathrm{e}^{\frac{\mathrm{i} \theta_{\ell}}{2}}
\end{array}
\right]
\end{equation}
for each $\ell = 1,\ldots,\frac{n+1}{2}$ and the last ones are
\begin{equation}\label{eq:3.6}
\sqrt{\frac{2}{n + 2}} \left[
\begin{array}{c}
  \mathrm{i} \, \mathrm{e}^{-\frac{\mathrm{i} \theta_{\ell}}{2}} - \mathrm{i} \, \mathrm{e}^{\frac{\mathrm{i} \theta_{\ell}}{2}} \\
  \mathrm{i} \, \overline{\omega} \, \mathrm{e}^{-\frac{\mathrm{i} \theta_{\ell}}{2}} - \mathrm{i} \, \overline{\omega}^{(n+1)} \, \mathrm{e}^{\frac{\mathrm{i} \theta_{\ell}}{2}} \\
  \mathrm{i} \, \overline{\omega}^{2} \, \mathrm{e}^{-\frac{\mathrm{i} \theta_{\ell}}{2}} - \mathrm{i} \, \overline{\omega}^{2(n+1)} \, \mathrm{e}^{\frac{\mathrm{i} \theta_{\ell}}{2}} \\
  \vdots \\
  \mathrm{i} \, \overline{\omega}^{n+1} \, \mathrm{e}^{-\frac{\mathrm{i} \theta_{\ell}}{2}} - \mathrm{i} \, \overline{\omega}^{(n+1)^{2}} \, \mathrm{e}^{\frac{\mathrm{i} \theta_{\ell}}{2}}
\end{array}
\right]
\end{equation}
for $\ell = \frac{n+1}{2}, \ldots, 1$. Since
\begin{gather}
\overline{\omega}^{k} \, \mathrm{e}^{-\frac{\mathrm{i} \theta_{\ell}}{2}} + \overline{\omega}^{(n+1)k} \, \mathrm{e}^{\frac{\mathrm{i} \theta_{\ell}}{2}} = \cos \left(\frac{\theta_{\ell}}{2} + \frac{2k\ell\pi}{n + 2} \right), \quad k=0,1,\ldots,n+1 \label{eq:3.7} \\[10pt]
\mathrm{i} \, \overline{\omega}^{k} \, \mathrm{e}^{-\frac{\mathrm{i} \theta_{\ell}}{2}} - \mathrm{i} \, \overline{\omega}^{(n+1)k} \, \mathrm{e}^{\frac{\mathrm{i} \theta_{\ell}}{2}} = \sin \left(\frac{\theta_{\ell}}{2} + \frac{2k\ell\pi}{n + 2} \right), \quad k=0,1,\ldots,n+1 \label{eq:3.8}
\end{gather}
we get that the entries of $\mathbf{P}_{n+2}$ are given by \eqref{eq:3.3b} which leads to \eqref{eq:3.3a}. Supposing a positive even integer $n$, Theorem 3.7 in \cite{Karner03} ensures
\begin{equation*}
\mathbf{A}_{n+2} = \mathbf{Q}_{n+2} \mathrm{diag} \left(\lambda_{0},\lvert\lambda_{1} \rvert,\ldots,\lvert \lambda_{\frac{n}{2}} \rvert,\lambda_{\frac{n}{2} + 1},-\lvert\lambda_{\frac{n}{2}} \rvert, \ldots,-\lvert \lambda_{1} \rvert \right) \mathbf{Q}_{n+2}^{\top}
\end{equation*}
where
\begin{equation*}
\mathbf{Q}_{n+2} = \boldsymbol{\Omega}_{n+2}^{\ast} \left[
\begin{array}{cc}
  1 & \mathbf{0} \\
  \mathbf{0} & \mathbf{S}_{n+1}
\end{array}
\right]
\end{equation*}
with $\boldsymbol{\Omega}_{n+2}$ given by \eqref{eq:3.2} and $\mathbf{S}_{n+1}$ the $(n+1) \times (n+1)$ matrix
\begin{equation*}
\mathbf{S}_{n+1} = \frac{\sqrt{2}}{2} \left[
\begin{array}{ccccccccc}
  \mathrm{e}^{-\frac{\mathrm{i} \theta_{1}}{2}} & 0 & \ldots & 0 & 0 & 0 & \ldots & 0 & \mathrm{i} \mathrm{e}^{-\frac{\mathrm{i} \theta_{1}}{2}} \\[5pt]
  0 & \mathrm{e}^{-\frac{\mathrm{i} \theta_{2}}{2}} & \ddots & \vdots & \vdots & \vdots & \iddots & \mathrm{i} \mathrm{e}^{-\frac{\mathrm{i} \theta_{2}}{2}} & 0 \\[5pt]
  \vdots & \ddots & \ddots & 0 & 0 & 0 & \iddots & \iddots & \vdots \\
  0 & \ldots & 0 & \mathrm{e}^{-\frac{\mathrm{i} \theta_{n/2}}{2}} & 0 & \mathrm{i} \mathrm{e}^{-\frac{\mathrm{i} \theta_{n/2}}{2}} & 0 & \ldots & 0 \\[5pt]
  0 & \ldots & 0 & 0 & \frac{2}{\sqrt{2}} & 0 & 0 & \ldots & 0 \\[5pt]
  0 & \ldots & 0 & \mathrm{e}^{\frac{\mathrm{i} \theta_{n/2}}{2}} & 0 & -\mathrm{i} \mathrm{e}^{\frac{\mathrm{i} \theta_{n/2}}{2}} & 0 & \ldots & 0 \\[5pt]
  \vdots & \iddots & \iddots & 0 & 0 & 0 & \ddots & \ddots & \vdots \\[5pt]
  0 & \mathrm{e}^{\frac{\mathrm{i} \theta_{2}}{2}} & \iddots & \vdots & \vdots & \vdots & \ddots & -\mathrm{i} \mathrm{e}^{\frac{\mathrm{i} \theta_{2}}{2}} & 0 \\[5pt]
  \mathrm{e}^{\frac{\mathrm{i} \theta_{1}}{2}} & 0 & \ldots & 0 & 0 & 0 & \ldots & 0 & -\mathrm{i} \mathrm{e}^{\frac{\mathrm{i} \theta_{1}}{2}}
\end{array}
\right].
\end{equation*}
The first column of $\mathbf{Q}_{n+2}$ has all its components equal to $1/\sqrt{n + 2}$. The next $n/2$ columns are given by \eqref{eq:3.5} for $\ell = 1,\ldots, \frac{n}{2}$ and the $n/2$ last ones are defined by \eqref{eq:3.6} for each $\ell = \frac{n}{2}, \ldots,1$; the $\left(\frac{n}{2} + 1 \right)$th column of $\mathbf{Q}_{n+2}$ is
\begin{equation*}
\sqrt{\frac{1}{n + 2}} \left[
\begin{array}{c}
  1 \\
  \overline{\omega}^{\frac{n}{2} + 1} \\
  \overline{\omega}^{2\left(\frac{n}{2} + 1 \right)} \\
  \vdots \\
  \overline{\omega}^{(n+1)\left(\frac{n}{2} + 1 \right)}
\end{array}
\right] = \sqrt{\frac{1}{n + 2}} \left[
\begin{array}{c}
  1 \\
  -1 \\
  (-1)^{2} \\
  \vdots \\
  (-1)^{n+1}
\end{array}
\right].
\end{equation*}
From identities \eqref{eq:3.7}, \eqref{eq:3.8} we obtain \eqref{eq:3.4a}. The proof is completed.
\end{proof}

The statement below is a decomposition for the matrices $\mathbf{H}_{n + 2}$ and plays a central role in the main results.

% LEMMA 2

\begin{lem}\label{lem2}
Let $n$ be a positive integer, $a,b,c$ real numbers, and $\omega$, $\lambda_{k}$, $\theta_{k}$, $k = 0,1,\ldots,n+1$ given by \eqref{eq:2.2}, \eqref{eq:2.3a}, \eqref{eq:2.3b}, respectively.

\medskip

\noindent \textnormal{(a)} If $n$ is odd,

\begin{equation}\label{eq:3.9}
\mathbf{x} = \sqrt{\frac{2}{n + 2}} \left[
\begin{array}{c}
  \frac{1}{\sqrt{2}} \\[5pt]
  \cos \left(\frac{\theta_{1}}{2} \right) \\[5pt]
  \vdots \\[5pt]
  \cos \left[\frac{\theta_{(n+1)/2}}{2} \right] \\[5pt]
  \sin \left[\frac{\theta_{(n+1)/2}}{2} \right] \\[5pt]
  \vdots \\[5pt]
  \sin \left(\frac{\theta_{1}}{2} \right)
\end{array}
\right], \quad
\mathbf{y} = \sqrt{\frac{2}{n + 2}} \left[
\begin{array}{c}
  \frac{1}{\sqrt{2}} \\[5pt]
  \cos \left(\frac{\theta_{1}}{2} - \frac{2 \pi}{n + 2} \right) \\[5pt]
  \vdots \\[5pt]
  \cos \left[\frac{\theta_{(n + 1)/2}}{2} - \frac{(n + 1) \pi }{n + 2} \right] \\[5pt]
  \sin \left[\frac{\theta_{(n + 1)/2}}{2} - \frac{(n + 1) \pi }{n + 2} \right] \\[5pt]
  \vdots \\[5pt]
  \sin \left(\frac{\theta_{1}}{2} - \frac{2 \pi}{n + 2} \right)
\end{array}
\right]
\end{equation}
then
\begin{equation*}
\mathbf{H}_{n+2}  = \mathbf{P}_{n+2} \left[\mathrm{diag} \left(\lambda_{0},\lvert\lambda_{1} \rvert,\ldots,\lvert \lambda_{\frac{n+1}{2}} \rvert,-\lvert\lambda_{\frac{n+1}{2}} \rvert, \ldots,-\lvert \lambda_{1} \rvert \right) - b \mathbf{x} \mathbf{x}^{\top} - a \mathbf{y} \mathbf{y}^{\top} \right] \mathbf{P}_{n+2}^{\top}
\end{equation*}
where $\mathbf{P}_{n+2}$ is the $(n+2) \times (n+2)$ matrix defined by \eqref{eq:3.3b}.

\medskip

\noindent \textnormal{(b)} If $n$ is even,
\begin{equation}\label{eq:3.10}
\mathbf{x} = \sqrt{\frac{2}{n + 2}} \left[
\begin{array}{c}
  \frac{1}{\sqrt{2}} \\[5pt]
  \cos \left(\frac{\theta_{1}}{2} \right) \\[5pt]
  \vdots \\[5pt]
  \cos \left(\frac{\theta_{n/2}}{2} \right) \\[5pt]
  \frac{1}{\sqrt{2}} \\[5pt]
  \sin \left(\frac{\theta_{n/2}}{2} \right) \\[5pt]
  \vdots \\[5pt]
  \sin \left(\frac{\theta_{1}}{2} \right)
\end{array}
\right], \quad
\mathbf{y} = \sqrt{\frac{2}{n + 2}} \left[
\begin{array}{c}
  \frac{1}{\sqrt{2}} \\[5pt]
  \cos \left(\frac{\theta_{1}}{2} - \frac{2 \pi}{n + 2} \right) \\[5pt]
  \vdots \\[5pt]
  \cos \left(\frac{\theta_{n/2}}{2} - \frac{n \pi }{n + 2} \right) \\[5pt]
  -\frac{1}{\sqrt{2}} \\[5pt]
  \sin \left(\frac{\theta_{n/2}}{2} - \frac{n \pi }{n + 2} \right) \\[5pt]
  \vdots \\[5pt]
  \sin \left(\frac{\theta_{1}}{2} - \frac{2 \pi}{n + 2} \right)
\end{array}
\right]
\end{equation}
then
\begin{equation*}
\mathbf{H}_{n+2}  = \mathbf{Q}_{n+2}\left[\mathrm{diag} \left(\lambda_{0},\lvert\lambda_{1} \rvert,\ldots,\lvert \lambda_{\frac{n}{2}} \rvert,\lambda_{\frac{n}{2} + 1},-\lvert\lambda_{\frac{n}{2}} \rvert, \ldots,-\lvert \lambda_{1} \rvert \right) - b \mathbf{x} \mathbf{x}^{\top} - a \mathbf{y} \mathbf{y}^{\top} \right] \mathbf{Q}_{n+2}^{\top}
\end{equation*}
where $\mathbf{Q}_{n+2}$ is the $(n+2) \times (n+2)$ whose the entries are given by \eqref{eq:3.4b}.
\end{lem}

\begin{proof}
We only prove (a) since (b) can be proven in the same way. Consider a positive odd integer $n$ and the following matrices
\begin{gather*}
    \mathbf{K}_{n+2} =
    \left[
    \begin{array}{cccc}
      b & 0 & \ldots & 0 \\
      0 & 0 & & \vdots \\
      \vdots & & \ddots & \vdots \\
      0 & \ldots & \ldots & 0
    \end{array}
    \right], \\[10pt]
    \mathbf{G}_{n+2} = \left[
    \begin{array}{cccc}
      0 & \ldots & \ldots & 0 \\
      \vdots & \ddots & & \vdots \\
      \vdots & & 0 & 0 \\
      0 & \ldots & 0 & a
    \end{array}
    \right].
\end{gather*}
From Lemma~\ref{lem1},
\begin{align*}
    \mathbf{P}_{n+2}^{\top} \mathbf{H}_{n+2} \mathbf{P}_{n+2} &= \mathbf{P}_{n+2}^{\top} \left(\mathbf{A}_{n+2} - \mathbf{K}_{n+2} - \mathbf{G}_{n+2} \right) \mathbf{P}_{n+2} \\
    &= \mathrm{diag} \left(\lambda_{0},\lvert\lambda_{1} \rvert,\ldots,\lvert \lambda_{\frac{n+1}{2}} \rvert,-\lvert\lambda_{\frac{n+1}{2}} \rvert, \ldots,-\lvert \lambda_{1} \rvert \right) - b \mathbf{x} \mathbf{x}^{\top} - a \mathbf{y} \mathbf{y}^{\top}
\end{align*}
where $\mathbf{P}_{n+2}$ is the matrix defined by \eqref{eq:3.3b}, $\mathbf{A}_{n+2}$ is the matrix \eqref{eq:3.1}, $\mathbf{x}$ is the first row of $\mathbf{P}_{n+2}$, i.e.
\begin{equation*}
\mathbf{x} =
\sqrt{\frac{2}{n + 2}} \left[
\begin{array}{c}
  \frac{1}{\sqrt{2}} \\[5pt]
  \cos \left(\frac{\theta_{1}}{2} \right) \\[5pt]
  \vdots \\[5pt]
  \cos \left[\frac{\theta_{(n+1)/2}}{2} \right] \\[5pt]
  \sin \left[\frac{\theta_{(n+1)/2}}{2} \right] \\[5pt]
  \vdots \\[5pt]
  \sin \left(\frac{\theta_{1}}{2} \right)
\end{array}
\right]
\end{equation*}
and $\mathbf{y}$ is the last row of $\mathbf{P}_{n+2}$,
\begin{equation*}
\mathbf{y} = \sqrt{\frac{2}{n + 2}} \left[
\begin{array}{c}
  \frac{1}{\sqrt{2}} \\[5pt]
  \cos \left[\frac{\theta_{1}}{2} + \frac{2 (n + 1) \pi}{n + 2} \right] \\[5pt]
  \vdots \\[5pt]
  \cos \left[\frac{\theta_{(n + 1)/2}}{2} + \frac{(n + 1)^{2}}{2} \cdot \frac{2 \pi }{n + 2} \right] \\[5pt]
  \sin \left[\frac{\theta_{(n + 1)/2}}{2} + \frac{(n + 1)^{2}}{2} \cdot \frac{2 \pi }{n + 2} \right] \\[5pt]
  \vdots \\[5pt]
  \sin \left[\frac{\theta_{1}}{2} + \frac{2(n + 1) \pi}{n + 2} \right]
\end{array}
\right] = \sqrt{\frac{2}{n + 2}} \left[
\begin{array}{c}
  \frac{1}{\sqrt{2}} \\[5pt]
  \cos \left(\frac{\theta_{1}}{2} - \frac{2 \pi}{n + 2} \right) \\[5pt]
  \vdots \\[5pt]
  \cos \left[\frac{\theta_{(n + 1)/2}}{2} - \frac{n + 1}{2} \cdot \frac{2 \pi }{n + 2} \right] \\[5pt]
  \sin \left[\frac{\theta_{(n + 1)/2}}{2} - \frac{n + 1}{2} \cdot \frac{2 \pi }{n + 2} \right] \\[5pt]
  \vdots \\[5pt]
  \sin \left(\frac{\theta_{1}}{2} - \frac{2 \pi}{n + 2} \right)
\end{array}
\right].
\end{equation*}
\end{proof}

\begin{proofTheorem1}
Consider a positive odd integer $n$, $\mathbf{x}, \mathbf{y}$ given by \eqref{eq:3.9}, and $d_{1} = \lambda_{0}, d_{2} = \lvert\lambda_{1} \rvert, \ldots,$ $d_{\frac{n+3}{2}} = \lvert\lambda_{\frac{n+1}{2}} \rvert, d_{\frac{n+5}{2}} = - \lvert\lambda_{\frac{n+1}{2}} \rvert,\ldots,d_{n+2} = - \lvert\lambda_{1} \rvert$. According to Lemma~\ref{lem2}, it should be noted that the matrix $\mathbf{H}_{n+2}$ and
\begin{equation}\label{eq:3.11}
\mathrm{diag} \left(d_{1},d_{2},\ldots,d_{n+2} \right) - b \mathbf{x} \mathbf{x}^{\top} - a \mathbf{y} \mathbf{y}^{\top}
\end{equation}
share the same eigenvalues. Let us adopt the notations of \cite{Anderson96} by denoting $\EuScript{S}(k,m)$ the collection of all $k$-element subsets of $\{1,2,\ldots,m \}$ written in increasing order; additionally, for any rectangular matrix $\mathbf{M}$, we shall indicate by $\det \left(\mathbf{M}[I,J] \right)$ the minor determined by the subsets $I = \left\{i_{1} < i_{2} < \ldots < i_{k} \right\}$ and $J = \left\{j_{1} < j_{2} < \ldots < j_{k} \right\}$. Setting

\begin{equation*}
\mathbf{X} =
\left[
\begin{array}{cc}
  -b \sqrt{\frac{2}{n + 2}} & -a \sqrt{\frac{2}{n + 2}} \\[10pt]
  -b \sqrt{\frac{2}{n + 2}} \cos \left(\frac{\theta_{1}}{2} \right) & -a \sqrt{\frac{2}{n + 2}} \cos \left(\frac{\theta_{1}}{2} - \frac{2\pi}{n + 2} \right) \\[2pt]
  \vdots & \vdots \\[5pt]
  -b \sqrt{\frac{2}{n + 2}} \cos \left[\frac{\theta_{(n+1)/2}}{2} \right] & -a \sqrt{\frac{2}{n + 2}} \cos \left[\frac{\theta_{(n+1)/2}}{2} - \frac{(n + 1)\pi}{n + 2} \right] \\[10pt]
  -b \sqrt{\frac{2}{n + 2}} \cos \left[\frac{\theta_{(n+1)/2}}{2} \right] & -a \sqrt{\frac{2}{n + 2}} \cos \left[\frac{\theta_{(n+1)/2}}{2} - \frac{(n + 1)\pi}{n + 2} \right] \\[5pt]
  \vdots & \vdots \\[5pt]
  -b \sqrt{\frac{2}{n + 2}} \cos \left(\frac{\theta_{1}}{2} \right) & -a \sqrt{\frac{2}{n + 2}} \cos \left(\frac{\theta_{1}}{2} - \frac{2\pi}{n + 2} \right)
\end{array}
\right]^{\top}
\end{equation*}
and
\begin{equation*}
\mathbf{Y} =
\left[
\begin{array}{cc}
  \sqrt{\frac{2}{n + 2}} & \sqrt{\frac{2}{n + 2}} \\[10pt]
  \sqrt{\frac{2}{n + 2}} \cos \left(\frac{\theta_{1}}{2} \right) & \sqrt{\frac{2}{n + 2}} \cos \left(\frac{\theta_{1}}{2} - \frac{2\pi}{n + 2}\right) \\[2pt]
  \vdots & \vdots \\[5pt]
  \sqrt{\frac{2}{n + 2}} \cos \left[\frac{\theta_{(n+1)/2}}{2} \right] & \sqrt{\frac{2}{n + 2}} \cos \left[\frac{\theta_{(n+1)/2}}{2} - \frac{(n + 1)\pi}{n + 2} \right] \\[10pt]
  \sqrt{\frac{2}{n + 2}} \cos \left[\frac{\theta_{(n+1)/2}}{2} \right] & \sqrt{\frac{2}{n + 2}} \cos \left[\frac{\theta_{(n+1)/2}}{2} - \frac{(n + 1)\pi}{n + 2} \right] \\[5pt]
  \vdots & \vdots \\[5pt]
  \sqrt{\frac{2}{n + 2}} \cos \left(\frac{\theta_{1}}{2} \right) & \sqrt{\frac{2}{n + 2}} \cos \left(\frac{\theta_{1}}{2} - \frac{2\pi}{n + 2} \right)
\end{array}
\right]^{\top},
\end{equation*}
we have from Theorem 1 of \cite{Anderson96} that $\zeta$ is an eigenvalue of \eqref{eq:3.11} if and only if
\begin{equation*}
1 + \sum_{k=1}^{n+2} \sum_{J \in \EuScript{S}\left(k,n+2 \right)} \sum_{I \in \EuScript{S}(k,2)} \frac{\det\left(\mathbf{X}[I,J] \right) \det\left(\mathbf{Y}[I,J] \right)}{\prod_{j \in J}(d_{j} - \zeta)} = 0
\end{equation*}
provided that $\zeta$ is not an eigenvalue of $\mathrm{diag} \left(d_{1}, \ldots,d_{n+2} \right)$. Since
\begin{equation*}
\begin{split}
1 + \sum_{k=1}^{n+2} & \sum_{J \in \EuScript{S}\left(k,n+2 \right)} \sum_{I \in \EuScript{S}(k,2)} \frac{\det\left(\mathbf{X}[I,J] \right) \det\left(\mathbf{Y}[I,J] \right)}{\prod_{j \in J}(d_{j} - \zeta)} = \\
&= 1 - b \sum_{k=1}^{n+2} \frac{\left[\mathbf{Y} \right]_{1,k}^{2}}{d_{k} - \zeta} - a \sum_{k=1}^{n+2} \frac{\left[\mathbf{Y} \right]_{2,k}^{2}}{d_{k} - \zeta} + a b \sum_{1 \leqslant k < \ell \leqslant n+2} \frac{\left(\left[\mathbf{Y} \right]_{1,k} \left[\mathbf{Y} \right]_{2,\ell} -  \left[\mathbf{Y} \right]_{1,\ell} \left[\mathbf{Y} \right]_{2,k} \right)^{2}}{(d_{k} - \zeta)(d_{\ell} - \zeta)}
\end{split}
\end{equation*}
we obtain \eqref{eq:2.4a}. Let $\mu_{1} \leqslant \mu_{2} \leqslant \ldots \leqslant \mu_{n+2}$ be the eigenvalues of $\mathbf{H}_{n+2}$ and $d_{\tau(1)} \leqslant d_{\tau(2)} \leqslant \ldots \leqslant d_{\tau(n+2)}$ be arranged in non-decreasing order by some bijection $\tau$ defined in $\{1,2, \ldots, n+2\}$. Thus,
\begin{equation}\label{eq:3.12}
\lambda_{\tau(k)} + \lambda_{\min} \left(- b \mathbf{x} \mathbf{x}^{\top} - a \mathbf{y} \mathbf{y}^{\top} \right) \leqslant \mu_{k} \leqslant \lambda_{\tau(k)} + \lambda_{\max} \left(- b \mathbf{x} \mathbf{x}^{\top} - a \mathbf{y} \mathbf{y}^{\top} \right)
\end{equation}
for each $k=1,2,\ldots,n+2$ (see \cite{Horn13}, page $242$). Using Miller's formula for the determinant of the sum of matrices (see \cite{Miller81}, page $70$), we can compute the characteristic polynomial of $- b \mathbf{x} \mathbf{x}^{\top} - a \mathbf{y} \mathbf{y}^{\top}$,
\begin{equation*}
\begin{split}
& \det \left(t \mathbf{I}_{n+2} + b \mathbf{x} \mathbf{x}^{\top} + a \mathbf{y} \mathbf{y}^{\top} \right) = \\
& \qquad = \left[1 + b t^{-1} \mathbf{x}^{\top} \mathbf{x} + a t^{-1} \mathbf{y}^{\top} \mathbf{y} + a b t^{-2} (\mathbf{x}^{\top} \mathbf{x}) (\mathbf{y}^{\top} \mathbf{y}) - a b t^{-2} (\mathbf{x}^{\top} \mathbf{y})^{2} \right] \det (t \mathbf{I}_{n+2}) \\
%& \qquad = t^{n} \left[t^{2} + a t \mathbf{x}^{\top} \mathbf{x} + b t \mathbf{y}^{\top} \mathbf{y} + a b (\mathbf{x}^{\top} \mathbf{x}) (\mathbf{y}^{\top} \mathbf{y}) - a b (\mathbf{x}^{\top} \mathbf{y})^{2} \right]
& \qquad = t^{n} \left[t^{2} + (a + b)t + a b \right]
\end{split}
\end{equation*}
because
\begin{gather*}
\mathbf{x}^{\top} \mathbf{x} = \frac{2}{n + 2} \left[\frac{1}{2} + \sum_{k = 1}^{\frac{n+1}{2}} \cos^{2} \left(\frac{\theta_{k}}{2} \right) + \sum_{k = 1}^{\frac{n+1}{2}} \sin^{2} \left(\frac{\theta_{k}}{2} \right) \right] = %\tfrac{2}{n + 2} \left[\tfrac{1}{2} + \frac{n + 1}{2} \right] =
1, \\
\mathbf{y}^{\top} \mathbf{y} = \frac{2}{n + 2} \left[\frac{1}{2} + \sum_{k = 1}^{\frac{n+1}{2}} \cos^{2} \left(\frac{\theta_{k}}{2}  - \frac{2k\pi}{n + 2} \right) + \sum_{k = 1}^{\frac{n+1}{2}} \sin^{2} \left(\frac{\theta_{k}}{2} - \frac{2k\pi}{n + 2} \right) \right] = 1,
\end{gather*}
and
\begin{align*}
\mathbf{x}^{\top} \mathbf{y} &= \frac{2}{n + 2} \left[\frac{1}{2} + \sum_{k = 1}^{\frac{n+1}{2}} \cos \left(\frac{\theta_{k}}{2} \right) \cos \left(\frac{\theta_{k}}{2}  - \frac{2k\pi}{n + 2} \right) + \sum_{k = 1}^{\frac{n+1}{2}} \sin \left(\frac{\theta_{k}}{2} \right) \sin \left(\frac{\theta_{k}}{2} - \frac{2k\pi}{n + 2} \right) \right] \\
&= \frac{2}{n + 2} \left[\frac{1}{2} + \sum_{k = 1}^{\frac{n+1}{2}} \cos \left(\frac{2k\pi}{n + 2} \right) \right] \\
&= \frac{2}{n + 2} \left\{\frac{1}{2} - \frac{1}{2} + \frac{\sin \left[\left(\frac{n + 1}{2} + \frac{1}{2} \right) \frac{2\pi}{n + 2} \right]}{2 \sin \left(\frac{2 \pi}{n + 2} \right)} \right\} \\
&= 0.
\end{align*}
Hence, $\mathrm{Spec} \left(- b \mathbf{x} \mathbf{x}^{\top} - a \mathbf{y} \mathbf{y}^{\top} \right) = \left\{0, -a, -b \right\}$ and \eqref{eq:3.12} yields \eqref{eq:2.4b}. The proof of the remaining assertion is performed in the same way and so will be omitted.
\end{proofTheorem1}

\begin{proofTheorem2}
Since both assertions can be proven in the same way, we only prove (a). Let $n$ be a positive odd integer, $a \neq 0$, and $\mu_{1},\mu_{2},\ldots,\mu_{n+2}$ the eigenvalues of $\mathbf{H}_{n+2}$. We can rewrite the matricial equation $(\mu_{m} \mathbf{I}_{n+2} - \mathbf{H}_{n+2}) \mathbf{z} = \mathbf{0}$ as
\begin{equation}\label{eq:3.13}
\mathbf{P}_{n+2} \left[\mu_{m} \mathbf{I}_{n+2} - \mathrm{diag} \left(\lambda_{0},\lvert\lambda_{1} \rvert,\ldots,\lvert \lambda_{\frac{n+1}{2}} \rvert,-\lvert\lambda_{\frac{n+1}{2}} \rvert, \ldots,-\lvert \lambda_{1} \rvert \right) + b \mathbf{x} \mathbf{x}^{\top} + a \mathbf{y} \mathbf{y}^{\top} \right] \mathbf{P}_{n+2}^{\top} \mathbf{z} = \mathbf{0}
\end{equation}
where $\mathbf{x}, \mathbf{y}$ are defined in \eqref{eq:3.9} and $\mathbf{P}_{n+2}$ the matrix whose the entries are given by \eqref{eq:3.3b}. Thus,
\begin{gather*}
\left[\mu_{m} \mathbf{I}_{n+2} - \mathrm{diag} \left(\lambda_{0},\lvert\lambda_{1} \rvert,\ldots,\lvert \lambda_{\frac{n+1}{2}} \rvert,-\lvert\lambda_{\frac{n+1}{2}} \rvert, \ldots,-\lvert \lambda_{1} \rvert \right) + b \mathbf{x} \mathbf{x}^{\top} + a \mathbf{y} \mathbf{y}^{\top} \right] \mathbf{w} = \mathbf{0}, \\
\mathbf{w} = \mathbf{P}_{n+2}^{\top} \mathbf{z}
\end{gather*}
that is,
\begin{equation*}
\mathbf{w} = \xi \left[\mu_{m} \mathbf{I}_{n+2} - \mathrm{diag} \left(\lambda_{0},\lvert\lambda_{1} \rvert,\ldots,\lvert \lambda_{\frac{n+1}{2}} \rvert,-\lvert\lambda_{\frac{n+1}{2}} \rvert, \ldots,-\lvert \lambda_{1} \rvert \right) + b \mathbf{x} \mathbf{x}^{\top} \right]^{-1}  \mathbf{y}
\end{equation*}
for $\xi \neq 0$ (see Theorem 5 of \cite{Bunch78}, page $41$) and
\begin{equation*}
\mathbf{z} = \xi \, \mathbf{P}_{n} \left[\mu_{m} \mathbf{I}_{n+2} - \mathrm{diag} \left(\lambda_{0},\lvert\lambda_{1} \rvert,\ldots,\lvert \lambda_{\frac{n+1}{2}} \rvert,-\lvert\lambda_{\frac{n+1}{2}} \rvert, \ldots,-\lvert \lambda_{1} \rvert \right) + b \mathbf{x} \mathbf{x}^{\top} \right]^{-1}  \mathbf{y}
\end{equation*}
is a nontrivial solution of \eqref{eq:3.13}. Thus, choosing $\xi = 1$ we conclude that the vector having components \eqref{eq:2.6} is an eigenvector of $\mathbf{H}_{n+2}$ associated to the eigenvalue $\mu_{m}$.
\end{proofTheorem2}

\section*{Acknowledgements}

This work is a contribution to the Project UID/GEO/04035/2013, funded by FCT - Funda\c{c}\~{a}o para a Ci\^{e}ncia e a Tecnologia, Portugal.


\begin{thebibliography}{99}

\bibitem{Anderson96} J. Anderson, A secular equation for the eigenvalues of a diagonal matrix perturbation, {\it Linear Algebra Appl.} {\bf246}, 49--70 (1996).

\bibitem{Bunch78} J.R. Bunch, C.P. Nielsen and D.C. Sorensen, Rank-one modification of the symmetric eigenproblem, {\it Numer. Math.} {\bf31}, 31--48 (1978).

\bibitem{Gutierrez14} J. Guti\'{e}rrez-Guti\'{e}rrez, Eigenvalue decomposition for persymmetric Hankel matrices with at most three non-zero anti-diagonals, {\it Appl. Math. Comput.} {\bf234}, 333--338 (2014).

\bibitem{Gutierrez16} J. Guti\'{e}rrez-Guti\'{e}rrez and M. Z\'{a}rraga-Rodr\'{\i}guez, Orthogonal diagonalization for complex skew-persymmetric anti-tridiagonal Hankel matrices, {\it Spec. Matrices} {\bf4}, 73--79 (2016).

\bibitem{Harville97} D.A. Harville, {\it Matrix Algebra From a Statistician's Perspective}, Springer-Verlag, New York (1997).

\bibitem{Horn13} R.A. Horn and C.R. Johnson, {\it Matrix Analysis} (second edition), Cambridge University Press, New York (2013).

\bibitem{Lita16} J. Lita da Silva, On anti-pentadiagonal persymmetric Hankel matrices with perturbed corners, {\it Comput. Math. Appl.} {\bf72}, 415--426 (2016).

%\bibitem{Lita18} J. Lita da Silva, Integer powers of anti-bidiagonal Hankel matrices, Indian J. Pure Appl. Math. 49(1) (2018) 87--98.

\bibitem{Karner03} H. Karner, J. Schneid and C.W. Ueberhuber, Spectral decompostion of real circulant matrices, {\it Linear Algebra Appl.} {\bf367}, 301--311 (2003).

\bibitem{Meyer00} C.D. Meyer, {\it Matrix Analysis Applied Linear Algebra}, Society for Industrial and Applied Mathematics, Philadelphia (2000).

\bibitem{Miller81} K.S. Miller, On the inverse of the sum of matrices, {\it Math. Mag.} {\bf54}(2),  67--72 (1981).

\bibitem{Rimas13a} J. Rimas, Integer powers of real odd order skew-persymmetric anti-tridiagonal matrices with constant anti-diagonals ($\mathrm{antitridiag}_{n}(a,c,-a)$, $a \in R \setminus \{0 \}$, $c \in R$), {\it Appl. Math. Comput.} {\bf219}, 7075--7088 (2013).

\bibitem{Rimas13b} J. Rimas, Integer powers of real even order anti-tridiagonal Hankel matrices of the form $\mathrm{antitridiag}_{n}(a,c,-a)$, {\it Appl. Math. Comput.} {\bf225}, 204--215 (2013).

\bibitem{Wu10} Honglin Wu, On computing of arbitrary positive powers for one type of anti-tridiagonal matrices of even order, {\it Appl. Math. Comput.} {\bf217}, 2750--2756 (2010).

\bibitem{Yin08} Qingxiang Yin, On computing of arbitrary positive powers for anti-tridiagonal matrices of even order, {\it Appl. Math. Comput.} {\bf203}, 252--257 (2008).

\end{thebibliography}
\end{document}